\documentclass[12pt,reqno,twoside]{article}
\usepackage{amssymb,amsfonts,amsthm,amsmath}
\usepackage{color}

\usepackage[pagebackref=false]{hyperref}
\usepackage{hyperref}
\usepackage{plain}
\usepackage{cite}
\usepackage[left=1 in,top=1 in,right=1 in,bottom=1 in]{geometry}

\numberwithin{equation}{section}

\def\beq{\begin{equation}}
\def\eeq{\end{equation}}
\def\beqs{\begin{equation*}}
\def\eeqs{\end{equation*}}

\newtheorem{theorem}{Theorem}[section]
\newtheorem{lemma}[theorem]{Lemma}
\newtheorem{proposition}[theorem]{Proposition}

\theoremstyle{definition}

\newtheorem{definition}[theorem]{Definition}

\newcommand{\esssup}{\mathop{\mathrm{ess\,sup}}}

\def\R{{\bf R}}

\def\C{{\bf C}}
\def\J{{\cal J}}
\def\DA{{\cal D}(A)}
\def\words#1{\quad\hbox{#1}\quad}
\def\wwords#1{\qquad\hbox{#1}\qquad}

\numberwithin{equation}{section}

\title{
	Spectral Filtering of Interpolant Observables \\
	for a Discrete-in-time Downscaling \\
    Data Assimilation Algorithm
}
\author{Emine Celik$^a$, Eric Olson$^a$ and Edriss S. Titi$^{b,c}$}

\date{February 1, 2019}

\begin{document}
{
    \catcode`\@=11
    \gdef\curl{\mathop{\operator@font curl}\nolimits}
}
\maketitle
\begin{center}
\textit{$^a$ Department of Mathematics and Statistics, University of Nevada, Reno\\
Reno, NV  89557, U. S. A.}\\
\textit{$^b$ Department of Mathematics, Texas A\&M University, 3368 TAMU\\
College Station, TX 77843-3368, U. S. A.}\\
\textit{$^c$ Department of Computer Science and Applied Mathematics,
	 The Weizmann \\ Institute of Science,
	Rehovot 76100, Israel}\\
	{\ }\\
Email addresses:  \texttt{ecelik@unr.edu, ejolson@unr.edu, titi@math.tamu.edu, edriss.titi@weizmann.ac.il}
\end{center}

\begin{abstract}
We describe a spectrally-filtered discrete-in-time downscaling data
assimilation
algorithm and prove, in the context of the two-dimensional
Navier--Stokes equations, that this algorithm works for a general
class of interpolants, such as those based on local spatial averages as
well as point measurements of the velocity.
Our algorithm is based on the classical technique of inserting
new observational data directly into the dynamical model as it is
being evolved over time, rather than nudging,
and extends previous results in which the observations were
defined directly in terms of an orthogonal projection onto
the large-scale (lower) Fourier modes.
In particular, our analysis does not require the interpolant to be
represented by an orthogonal projection,
but requires only the interpolant to satisfy
a natural approximation of the identity.
\end{abstract}

\section{Introduction}\label{intro}
The goal of data assimilation is to optimally combine known
information about the dynamics of a solution with low-resolution
observational measurements of that solution over time to create better
and better approximations of the current state.
While model error in the dynamics and measurement error in the
observations are significant issues with practical data assimilation,
we consider here the error-free case in order to study the
role played by spatial filtering.
In particular, even if the observations are error free,
in certain geophysical models
they can contain high-frequency spillover and gravity waves
which need to be controlled in order for the data assimilation
to perform well.
Additional issues arise because commonly used filtering
techniques can lead to non-orthogonal interpolants.
These issues are the focus of the current paper.
Our results extend the work of Hayden, Olson and Titi
\cite{Hayden2011} on discrete-in-time data assimilation
from the case where the low-resolution observations are given by
projection onto the low Fourier modes
to both the first and second type of general interpolant
observables that appear in Azouani, Olson and Titi \cite{Azouani2014},
see also Bessaih, Olson and Titi~\cite{Bessaih2015}.
To make this extension, we apply a spectral filter based on the
Stokes operator to the interpolant observables and call the
new method spectrally-filtered discrete-in-time downscaling data assimilation.
It is worth noting that much of advances in the accuracy of present
day weather forecasting have come from better filtering techniques,
see for example Budd, Freitag and Nichols \cite{Budd2011}.
From this point of view, the analytic results presented here for
spectral filtering may be seen as a first step towards a
rigorous analysis of more complicated methods.

An alternative algorithm for discrete-in-time data assimilation
based on nudging was recently studied by Foias, Mondaini and Titi
in \cite{Foias2016}.
In that work it was shown that nudging works for interpolants of 
what is known by now as type-I, such as those which are based
on local coarse spatial scale volume elements measurements without any additional
filtering---the dissipation provided by the Navier--Stokes equations
themselves is sufficient; however,
a similar treatment for type-II interpolant observables is missing.
The algorithm studied here is based on the classical technique
of inserting the observational data directly into the model as it
is evolved forward in time, see for example Daley \cite{Daley1991}
and references therein.
When inserting the data directly into the model,
the need for filtering becomes more evident.
Moreover, by developing a spectrally-filtered algorithm we are
also able to handle type-II interpolant observable.
Although
it is likely a similar technique could be applied to a
nudging algorithm to handle type-II interpolant observables,
we do not pursue that line of analysis here,
but will be reported in future work.

The two-dimensional incompressible Navier--Stokes equations are given by
\begin{plain}\begin{equation}\label{navierstokes}
	{\partial U\over\partial t}-\nu\Delta U
		+\nabla P + (U\cdot\nabla) U = f,\qquad
	\nabla\cdot U = 0.
\end{equation}\end{plain}%
Following Constantin and Foias~\cite{Constantin1988},
Foias, Manley, Rosa and Temam~\cite{Foias2001},
Robinson~\cite{Robinson2001} and
Temam~\cite{Temam1983}, and in order to simplifying our presentation and fix ideas,
we consider flows on the domain $\Omega=[0,L]^2$ equipped
with periodic boundary conditions.
Let ${\cal V}$ be the set of all divergence-free $L$-periodic
trigonometric polynomials with zero spatial averages,
$V$ be the closure of ${\cal V}$ in $H^1(\Omega,\R^2)$,
$V^*$ be the dual of $V$, and
$P_\sigma$ be the orthogonal projection of $L^2(\Omega;\R^2)$
onto $H$, where
$H$ is the closure of ${\cal V}$ in $L^2(\Omega,\R^2)$.
Define $A\colon V\to V^*$ and $B\colon V\times V\to V^*$
to be the unique continuous extensions
for $u,v\in{\cal V}$
of the operators given by
$$
    Au=-P_\sigma \Delta u\qquad\hbox{and}\qquad
    B(u,v)=P_\sigma (u\cdot\nabla v).
$$
Remark that in periodic case $A=-\Delta$, thus, the two-dimensional incompressible
Navier--Stokes equations may be written as
\begin{plain}\begin{equation}\label{2dns}
    {dU\over dt}+\nu AU+B(U,U)=f
\end{equation}\end{plain}%
with initial condition $U_0\in V$, at time $t=t_0$.
Here
$\nu>0$ is the kinematic viscosity, and the body force
$f\in L^\infty([t_0,\infty);H)$
is taken to be divergence free, but possibly time dependent.

When the force is time independent, as shown in any of the aforementioned
references, equations (\ref{2dns}) are well posed with
unique regular solutions depending continuously on the initial
conditions and which exist for all time, $t\ge t_0$.
The case when the force depends on time is somewhat more delicate
and we shall place further assumptions on $f$ in Section \ref{secprelim},
see also Appendix \ref{apriori},
to ensure the resulting solutions have enough regularity
for the subsequent analysis.
In either case, we define the semi-process $S$ as the
solution operator
that maps initial conditions into their subsequent time
evolution by $S(t,t_0;U_0)=U(t)$ for all $t\ge t_0$.

We now describe the general interpolant observables to which our
results will apply.
These interpolants are inspired by the modes, nodes and
volume elements of Jones and Titi \cite{Jones1993},
see also Foias and Titi \cite{Foias1991},
and are equivalent to
the first and second types of general interpolant observables
that appear in \cite{Azouani2014}, see also \cite{Bessaih2015} and the
general framework presented in Cockburn, Jones and Titi \cite{Cockburn1997}.
In particular, we state
\begin{definition}\label{interpolants}
A linear operator ${I_h}\colon V\to L^2$ is said
to be a {\it type-I interpolant observable\/} if there
exists $c_1>0$ such that
\beq\label{typeone}
\|U-{I_h}U\|_{L^2}^2\le c_1h^2\|U\|^2
\wwords{for all} U\in V.
\eeq
A linear operator ${I_h}\colon \DA\to L^2$ is said to be a
{\it type-II interpolant observable\/} if
\beq\label{typetwo}
\|U-{I_h}U\|_{L^2}^2 \le c_1h^2\big(\|U\|^2+h^2|AU|^2\big)
\wwords{for all} U\in \DA.
\eeq
\end{definition}

Here $\DA=H^2(\Omega)\cap V$ is the domain of $A$ viewed as an operator into $L^2$.
Specifically, in terms of Fourier modes, let
$$
	H = \left\{\,
		\sum_{k\in\J}\widehat U_ke^{ik\cdot x}
		\words{:}
		\widehat U_k\in\C^2,\quad
		\widehat U_k^* = \widehat U_{-k},\quad k\cdot \widehat U_k=0
		\words{and}
		\sum_{k\in\J}|\widehat U_k|^2 <\infty
	\,\right\},
$$
where
$$
\J=\left\{ \frac{2\pi}{L}(n_1,n_2):
	n=(n_1,n_2)\in\mathbb{Z}^2\backslash\{(0,0)\}  \right\}.
$$
For notational convenience assume $\widehat U_0=0$ even
though this coefficient doesn't enter into the above characterization
of $H$.
We employ the notations
$$
	|U|=\|U\|_0,\qquad
	\|U\|=\|U\|_1
	\wwords{and}
	|AU|=\|U\|_2$$
where
\begin{equation}\label{alphanorm}
\|U\|_\alpha^2=L^2\sum_{k\in\J}|k|^{2\alpha}|\widehat U_k|^2
\wwords{when}
U=\sum_{k\in\J} \widehat U_k e^{ik\cdot x}.
\end{equation}
Further define
$V_\alpha= \{\, U\in H : \|U\|_\alpha<\infty\,\}$.  Consequently
$\DA=V_2$ and $V=V_1$.

In Definition \ref{interpolants} we note that
$h$ is a length scale corresponding to the observation resolution
and $c_1$ is a dimensionless constant.
For example, suppose nodal measurements of the velocity are given by
$$
	\big(U(x_1),U(x_2),\ldots,U(x_d)\big),
$$
where $x_i\in\Omega$ have been chosen in such a way that
$$
		\sup_{x\in\Omega} \inf \big\{\, \|x-x_j\|: j=1,2,\ldots,d\,\big\}
		\le h.
$$
Then
\begin{plain}$$
	{I_h(U)(x)=
    \sum_{j=1}^{d} U(x_i) \widetilde\chi_j(x)}
\wwords{where}
	{\widetilde\chi_j(x)= \chi_j(x)
		- {1\over |\Omega|}\int_\Omega \chi_j}
$$\end{plain}%
with
\begin{plain}
$$
	\chi_j(x)=
\cases{
	1 & if $\|x-x_j\|<\|x-x_i\|$ for all $i\ne j$\cr
	0 & otherwise
}
$$\end{plain}%
is a type-II interpolant observable.

It is worth reflecting that the type-II interpolant observable
described above naturally results in a piecewise constant vector 
field which is discontinuous.
Although $I_h(U)\in L^2$ as required, the Fourier transform of
the resulting vector field possesses a significant high-frequency 
component due to the discontinuities.  
A similar interpolant was considered in \cite{Gesho2016}
for numerical simulations of a data-assimilation method based on nudging.
Those computations show that the adverse effects of the 
high-frequency spill over which result from the spatial discontinuities 
can be mitigated by appropriate convolution with a smoothing kernel.
The spectral filtering considered in this work also removes
the high-frequency component in spatial Fourier representation while 
enjoying additional approximation properties useful for the 
analysis of the resulting data assimilation algorithm.

We now introduce the spectrally-filtered discrete-in-time data
assimilation algorithm which forms the focus of our study.
Let $P_\lambda\colon H\to H$ be the orthogonal projection
onto the Fourier modes with wave numbers $k$ such
that $|k|^2\le \lambda$ given by
$$
	P_\lambda U = \sum_{|k|^2\le \lambda} \widehat U_k e^{ik\cdot x}.
$$
and let $Q_\lambda=I-P_\lambda$ be the
orthogonal complement of $P_\lambda$.
Now, given $\lambda>0$ and $I_h$ define
\begin{equation}\label{filtered}
	{J=P_\lambda P_\sigma  I_h}
\wwords{and}
	E=I-J.
\end{equation}

Note, although no additional orthogonality or
regularity properties other than those appearing in Definition
\ref{interpolants} have been assumed
on $I_h$, the above spectral filtering yields an
operator~$J$ which is nearly orthogonal and has
a range contained in $\DA$.
The downscaling data assimilation algorithm studied in this
paper may now be stated as

\begin{definition}\label{algorithm}
Let $U$ be an exact solution of \eqref{2dns} which evolves according
to dynamics given by the semi-process $S$.
Let $t_n=t_0+n\delta$ be a sequence of times for which
partial observations of $U$ are interpolated by $I_h$.
Then the approximating solution $u$ given by
\begin{plain}
$$
\left\{
\eqalign{
    u_0&=J U(t_0)\cr
    u_{n+1}&= E S(t_{n+1},t_n;u_n)+J U(t_{n+1})\cr
    u(t)&=S(t,t_n;u_n)\qquad\hbox{for}\qquad t\in[t_n,t_{n+1})\cr
}\right.
$$
is what we shall call
{\it spectrally-filtered discrete-in-time downscaling data assimilation}.
\end{plain}%
\end{definition}
We stress that only the spectrally filtered low-resolution observations
of the exact solution represented by $J U(t_n)$ for $t_n\le t$ are
used to construct the approximating solution $u$ at time~$t$.
Since we assume the dynamics governing the evolution of
$U$ to be known, then exact knowledge of the initial condition $U(t_0)=U_0$
would, in theory, obviate the need for data assimilation at subsequent times.
Of course, knowing the exact dynamics and being able to practically
compute with them are two different things.
Although not the focus of the present research,
the algorithm stated above may also be used
to stabilize the growth of numerical error.
Putting such numerical considerations aside,
we view the data assimilation algorithm given in
Definition~\ref{algorithm} as a way of improving
estimates of the unknown state of $U$ at time $t$ by means of known
dynamics and a time-series of low-resolution observations.

Intuitively, at each time $t_{n+1}$ a new measurement is used to
kick the approximating solution towards the exact solution by constructing
of an improved approximation of the current state $u_{n+1}$ which may be
seen as a combination of a prediction
based on the previous approximation
and a correction based on the observation.
This improved approximation then serves as an initial condition
from which to further evolve the approximating solution.
Since $JU(t_{n+1})$ is supported on a finite number of Fourier modes,
the regularity of $u_{n+1}$ is determined by
$ES(t_{n+1},t_n;u_n)$.
For type-I interpolant observables our working assumptions described
after Proposition \ref{attbnd} in Section \ref{intro} shall
imply that $u_{n+1}\in V$ and for type-II that $u_{n+1}\in\DA$.

Although we have taken the sequence of observation times $t_n$ to
be equally spaced, intuitively one might imagine
for a suitably small value of $\delta$ that
it would be sufficient for
\begin{equation}\label{imagine}0<t_{n+1}-t_n\le\delta,
\wwords{with} t_n\to\infty,\words{as} n\to\infty.
\end{equation}
Our analysis, however, makes use of a minimum distance between
$t_{n+1}$ and $t_n$ as well as the maximum.
Measurements need to be inserted frequently enough to overcome
the tendency for two nearby solutions to drift apart, while
at the same time the possible lack of orthogonality in our
general interpolant observables means measurements should not
be inserted too frequently.
Specifically, we need to have enough time to elapse between each insertion to allow enough time for the use of the dynamics of the equation, i.e. integrating the Navier-Stokes equations for long enough time to correct the high modes. Our algorithm consists of two steps:
\textit{Step 1.} Inserting the coarse spatial scale measurements.
\textit{Step 2.} Integrating the Navier-Stokes equations for short time, but not too short, to recover and correct the missing high modes, i.e. the fine spatial scales of the solutions.
 Preliminary numerical simulations further indicate this requirement 
is likely physical and not merely a technical condition used 
by our analysis.
Given times $t_n$ that satisfy (\ref{imagine})
it would
be straightforward to construct a subsequence of observations $t'_n$
such that $\delta/2 < t'_{n+1}-t'_n \le 2\delta$
and obtain results similar to the ones presented here.
We leave such a refinement to the reader.

Note that the algorithm described above reduces to the discrete
data assimilation method studied in \cite{Hayden2011}
by taking $I_h=P_\lambda$.
In particular, when the interpolant observable itself is
given by an orthogonal projection onto the large-scale
Fourier modes.
In this work $I_h$ can by any interpolant operator satisfying
Definition \ref{interpolants}.
Carefully adjusting
the relationship between $h$ and $\lambda$ then allows us to
prove our main result, stated as
\begin{theorem}\label{mainresult}
Let $U$ be a solution to the two-dimensional incompressible
Navier--Stokes equations (\ref{2dns}) and $u(t)$, for $t\ge t_0$
be the process given by
Definition \ref{algorithm}.
Then, for every $\delta>0$ there exists $h>0$ and $\lambda>0$
depending only on $c_1$, $f$, $\nu$
such that
$$
	|u(t)-U(t)|\to 0,\words{exponentially in time, as} t\to\infty.
$$
Here $c_1$ is the constant
in Definition \ref{interpolants}
given by the general interpolant observables.
\end{theorem}
\noindent

Since we have assumed the observational measurements to be noise-free
and that the exact solution evolves according to known dynamics, it is
natural to obtain a result in which the difference between the exact
solution $U$ and the approximation $u$ converges to zero over time.
We further remark that if by chance $u(t_n)=U(t_n)$
at any of the data assimilation steps,
then $u(t)=U(t)$ for all $t\ge t_n$.
In particular, if somehow $U_0$ is known exactly and we take $u_0=U_0$
as the first step of Definition \ref{algorithm},
then $u(t)=U(t)$ for all $t\ge t_0$.

This paper is organized as follows.
In section 2 we set our notation, recall some facts about the
Navier--Stokes equations and prove some preliminary results
regarding the spectrally-filtered interpolant observables that
will be used in our subsequent analysis.
Section 3 proves our main result for type-I
interpolant observables while section 4 treats the case
of type-II interpolant observables.
We finish with some concluding remarks concerning the
dependency of $h$ and $\lambda$ on $\delta$ and the other
physical parameters in the system.

\section{Preliminaries}\label{secprelim}
We begin by recalling some inequalities.
Writing the smallest eigenvalue of the Stokes operator $A$ as
$\lambda_1=(2\pi/L)^2$ we have the
Poincar\'e inequalities
\begin{equation}\label{poincareV}
    \lambda_1 |U|^2\le \|U\|^2\wwords{for} U\in V
\end{equation}
and
\begin{equation}\label{poincareDA}
	\lambda_1^2 |U|^2\le \lambda_1 \|U\|^2\le |AU|^2
\wwords{for}
	U\in\mathcal{D}(A).
\end{equation}
An advantage of using the projection $P_\lambda$ in our
data assimilation algorithm,
rather than a different type of spatial filtering,
is that this directly leads to
improved Poincar\'e inequalities and reverse
inequalities which are, respectively, given by
\beq\label{pimproved}
	\lambda |Q_\lambda U|^2 \le\|Q_\lambda U\|^2
\wwords{and}
    \lambda^2 |Q_\lambda U|^2
	\le \lambda \|Q_\lambda U\|^2\le |A Q_\lambda U|^2
\eeq
and
\beq\label{preverse}
\|P_\lambda U\|^2\le\lambda |P_\lambda U|^2
\wwords{and}
|AP_\lambda U|^2\le\lambda^2 |P_\lambda U|^2.
\eeq
All of the inequalities given in \eqref{poincareV},
\eqref{poincareDA}, \eqref{pimproved}
and \eqref{preverse} may easily be verified via Fourier series.
We also recall Agmon's
inequality \cite{Agmon2010} (see also \cite{Constantin1988}) as
\begin{equation}\label{agmon}
\|U\|_{L^\infty}\le C|U|^{1/2}|AU|^{1/2}.
\end{equation}
Here $C$ is a dimensionless constant depending only
on the domain $\Omega$.

As mentioned in the introduction, the spectrally filtered
interpolation operator $J$ given by (\ref{filtered})
possesses approximate orthogonality and regularity properties
that the original interpolant observable $I_h$ may fail to have.
We summarize these properties in
\begin{proposition}\label{Ebounds}
Let $c_1$ be the dimensionless constant appearing
in Definition \ref{interpolants}.
For type-I interpolant observables setting
$\varepsilon=c_1\lambda h^2$ yields
$$
	|EU|^2\le \lambda^{-1}(1+\varepsilon)\|U\|^2
\words{and}
	\|EU\|^2\le (1+\varepsilon)\|U\|^2,
\words{for every} U\in V.
$$
For type-II interpolant observables
setting $\varepsilon=c_1\lambda_1^{-1}\lambda^2h^2(1+\lambda_1h^2)$
yields
$$
	|EU|^2\le(\lambda\lambda_1)^{-1}(1+\varepsilon) |AU|^2,\qquad
	\|EU\|^2\le \lambda_1^{-1}(1+\varepsilon) |AU|^2
$$
and
$$|AEU|^2\le (1+\varepsilon) |AU|^2
\wwords{for every} U\in \mathcal{D}(A).$$
\end{proposition}

\begin{proof}[Proof of Proposition \ref{Ebounds}]
Estimate $|EU|$ for type-I interpolant observables as
\begin{align*}
	|EU|^2&=|U-JU|^2=|U-P_\lambda U+P_\lambda U-P_\lambda P_\sigma I_h U|^2\\
	&\le |Q_\lambda U|^2+|P_\lambda(U-P_\sigma I_hU)|^2
	\le \lambda^{-1}\|U\|^2+ |U-P_\sigma I_hU|^2\\
	&=\lambda^{-1}\|U\|^2+ |P_\sigma(U- I_hU)|^2\le
	\lambda^{-1}\|U\|^2+ \|U- I_hU\|_{L^2}^2\\
	&\le \lambda^{-1}\|U\|^2+c_1h^2\|U\|^2
		=\left(\lambda^{-1}+c_1h^2\right)\|U\|^2.
\end{align*}
From the definition of $\varepsilon$ it follows that
\beq\label{EU}
	|EU|^2\le  \lambda^{-1}\big(1+\varepsilon\big)\|U\|^2.
\eeq
Similarly bound $\|EU\|$ as
\begin{align*}
	\|EU\|^2 &=\|U-JU\|^2=\|Q_\lambda U\|^2+\|P_\lambda(U-P_\sigma I_hU)\|^2\\
	&\le\|U\|^2+\lambda|U-P_\sigma I_hU|^2
	\le\|U\|^2+\lambda c_1h^2\|U\|^2
	\le (1+\varepsilon)\|U\|^2.
\end{align*}
Now, estimate $|EU|$ for type-II interpolant observables as
\beq
\begin{aligned}
	|EU|^2
	&=|U-P_\lambda P_\sigma I_hU|^2
	=|U-P_\lambda U+P_\lambda U-P_\lambda P_\sigma I_hU|^2\\
	&=|U-P_\lambda U|^2+|P_\lambda P_\sigma(U-I_hU)|^2
	\le |Q_\lambda U|^2 + \|U- I_h U\|_{L^2}^2\\
	&\le |Q_\lambda U|^2+c_1h^2\left(\|U\|^2+h^2|AU|^2\right)\\
	&\le \left( \frac 1{\lambda}+c_1h^2 \right)\|U\|^2+c_1h^4|AU|^2\\
	&\le \left[\frac{1}{\lambda_1}\left( \frac 1{\lambda}+c_1h^2 \right)
		+c_1h^4\right]|AU|^2.
\end{aligned}
\eeq
Setting $\varepsilon=c_1\lambda_1^{-1}\lambda^2h^2(1+\lambda_1h^2)$
yields that
$$
	|EU|^2\le (\lambda\lambda_1)^{-1}(1+\varepsilon)|AU|^2.
$$
Next, estimate $\|EU\|$ as
\beq
\begin{aligned}
	\|EU\|^2
	&=\|U-P_\lambda P_\sigma I_hU\|^2
	=\|U-P_\lambda U+P_\lambda U-P_\lambda P_\sigma I_hU\|^2\\
	&=\|U-P_\lambda U\|^2+\|P_\lambda P_\sigma (U-I_hU)\|^2\\
	&=\|Q_\lambda U\|^2+\lambda|P_\sigma(U-I_hU)|^2\\
	&\le\|Q_\lambda U\|^2+\lambda\|U-I_hU\|_{L^2}^2\\
	&\le \lambda^{-1}|AU|^2+c_1\lambda h^2
		\big(\|U\|^2+h^2|AU|^2\big)\\
	&\le \lambda^{-1}|AU|^2+c_1\lambda h^2
		\big(\lambda_1^{-1}|AU|^2+h^2|AU|^2\big)\\
	&\le \big(\lambda^{-1}+c_1\lambda_1^{-1}\lambda
		h^2(1+\lambda_1 h^2)\big)|AU|^2\\
	&\le \lambda_1^{-1}\big(1+c_1 \lambda h^2(1+\lambda_1 h^2)\big)|AU|^2
	\le \lambda_1^{-1}(1+\varepsilon)|AU|^2,\\
\end{aligned}
\eeq
and finally $|AEU|$ as
\beq
\begin{aligned}
	|AEU|^2
	&=|Q_\lambda AU|^2+|AP_\lambda (U-P_\sigma  I_hU)|^2\\
	&\le |A U|^2+\lambda^2|P_\sigma(U-I_hU)|^2\\
	&\le |A U|^2+\lambda^2\|U-I_hU\|_{L^2}^2\\
	&\le |AU|^2+c_1\lambda^2 h^2
		\big(\|U\|^2+h^2|AU|^2\big)\\
	&\le \big(1+c_1\lambda_1^{-1}\lambda^2
		h^2(1+\lambda_1 h^2)\big)|AU|^2
	= (1+\varepsilon)|AU|^2.\\
\end{aligned}
\eeq
This completes the proof of the proposition.
\end{proof}

Our analysis will make use of {\it a priori\/} bounds on
the solution $U$ of \eqref{2dns}.
If $f\in H$ is time independent,
such bounds can be inferred from
bounds on the global attractor.
For example, Propositions~12.2 and 12.4 in Robinson \cite{Robinson2001}
may be stated as
\begin{proposition}\label{attbnd}
If $f\in H$ is time independent, then there are absorbing sets
in $H$, $V$ and $\DA$ of radiuses
$\rho_H$, $\rho_V$ and $\rho_A$, respectively,
depending only on $|f|$, $\Omega$ and $\nu$ such that
for every $U_0\in H$ there is a time $t_A$ depending only
on $|U_0|$ for which
\begin{align}\label{rhoA}
| U(t) |\le \rho_H,\quad
\| U(t) \|\le \rho_V
\words{and}
|AU(t)|\le \rho_A
\words{for all} t\ge t_A.
\end{align}
Moreover,
\begin{equation}\label{Aint}
\int_t^{t+\delta} |AU|^2 \le
	\Big({1\over \nu}+{\delta\lambda_1\over 2}\Big)
	\rho_V^2
		\words{for all} t\ge t_A.
\end{equation}
\end{proposition}
\noindent
Similar bounds may be found in
Temam \cite{Temam1983} and
Constantin and Foias \cite{Constantin1988}.
The best estimate of $\rho_A$ to date appears in \cite{Foias2015}.
Before considering the case when $f$ depends on time, we further
note when $f\in V$ is time independent that
the bounds in (\ref{rhoA}) are finite for $t>t_0$.
Moreover, (\ref{Aint}) is finite and
\begin{equation}\label{vortpose}
	\int_t^{t+\delta} \|A U\|^2<\infty
\wwords{for all} t\ge t_0.
\end{equation}
We remark that estimate (\ref{vortpose}) follows as a particular case of the proof presented in Appendix \ref{apriori} for the time-dependent forcing term, see discussion below.

When $f\in L^{\infty}([t_0,\infty),H)$
depends on time, the resulting solution $U$ does not
automatically satisfy the $\rho_A$ bound in (\ref{rhoA})
nor the finiteness condition (\ref{vortpose}).
In the case of
type-I interpolant observables
the remaining bounds given by $\rho_H$ and $\rho_V$
are sufficient for our analysis.
However, for type-II interpolant observables
we need $\rho_A$ as well as the
finiteness condition (\ref{vortpose}).
These bounds may be obtained in a number of different ways.
For example, one could assume that
$f\in L^{\infty}([t_0,\infty),V)$ and
$df/dt\in L^{\infty}([t_0,\infty),V^*)$.
For details see Appendix \ref{apriori}.

Our analysis shall be made under the working
assumption that $\rho_H$, $\rho_V$ and $\rho_A$ are known when
needed and
that the unknown initial condition $U_0$ in \eqref{2dns}
comes from a long-time evolution prior to time $t_0$.
Thus, we assume $t_0\ge t_A$ and in particular that
the bounds \eqref{rhoA}, \eqref{Aint} and \eqref{vortpose}
hold, in fact, for $t\ge t_0$ regardless of whether $f$
depends on time or not.
For other initial conditions we further suppose that the norms
and time integrals appearing in all the above bounds are
at least finite when $t>t_0$.
We now state a standard result
concerning the finite-time continuous
dependence on initial conditions
for solutions
to the two-dimensional
incompressible Navier-Stokes equations.

\begin{theorem}\label{beta}
Under the working assumptions given above,
there exists $\beta>0$ depending only on {$|f|$, $L$ and
$\nu$} such that the free-running solution satisfies
\beq\label{wttn}
	|U(t)-S(t,t_*;{u_*})|^2 \le
		e^{\beta(t-t_*)}|U(t_*)-{u_*}|^2
\words{for} t\ge t_*\words{and}u_*\in V.
\eeq
\end{theorem}
\noindent
We remark that the above continuity result is obtained from the
first Lyapunov exponent, which reflects the instability in turbulent
flows.
Thus, the constant $\beta$ in Theorem \ref{beta}
is very large but uniform for $u_*\in V$.
The fact that $\beta$ does not depend on $u^*$ is a fact we
shall make salient use of in our subsequent analysis.

We recall that the bilinear term $B$ has the algebraic property that
\begin{equation}
    \big\langle B(u,v),w\big\rangle = -\big\langle B(u,w),v\big\rangle
    \label{balg}
\end{equation}
for $u,v,w\in V$, and consequently the orthogonality property that
\begin{equation}
    \big\langle B(u,w),w\big\rangle =0.
    \label{borth}
\end{equation}
Here the pairing $\langle \cdot,\cdot\rangle$ denotes the
dual action of $V^*$ on $V$.  Details may be found, e.g., in
\cite{Constantin1988}, \cite{Foias2001}, \cite{Robinson2001}
and \cite{Temam1983}.
In the case of periodic boundary conditions the bilinear term
possesses the additional orthogonality property
\begin{equation}
    \big(B(w,w),Aw\big)=0,
\qquad\hbox{for every}\quad w\in\DA;
    \label{balgp}
\end{equation}
and consequently one has 
\begin{equation}
    \big(B(u,w),Aw\big)
    +\big(B(w,u),Aw\big)
    =-\big(B(w,w),Au\big), \qquad\hbox{for every}\quad u, w\in \DA.
    \label{borthp}
\end{equation}

We further recall some well-known bounds on the non-linear term which
appear in \cite{Constantin1988}, \cite{Temam1983}, \cite{Temam1997} and
specifically as Proposition 9.2 in \cite{Robinson2001}.
\begin{proposition}
One has
\beq\label{Best0}
|(B(u,v),w)|\le \|u\|_{L^\infty}\|v\| |w|,
\eeq
where $u\in L^\infty, $ $v\in V$ and $w\in H$.
If $u, v, w \in V$ then
\beq\label{Best}
|(B(u,v),w)|\le c|u|^{1/2}\|u\|^{1/2}\|v\| |w|^{1/2}\|w\|^{1/2},
\eeq
and if $u\in V$ $v\in \DA$, and $w\in H$,
\beq\label{Best1}
|(B(u,v),w)|\le c|u|^{1/2}\|u\|^{1/2}\|v\|^{1/2}|Av|^{1/2} |w|.
\eeq
Here $c$ is an absolute non-dimensional constant.
\end{proposition}

\section{Type-I Interpolant Observables}
In this section we treat the case when $I_h$ is a type-I
interpolant observable.
While type-I interpolant observables are also of type II,
the bounds we obtain in treating these two cases separately
are sharper.
In addition, the proof for in the type-I case is simpler
and serves as a framework to help understand the
more complicated type-II case teated in the subsequent section.
From Definition \ref{algorithm} it follows that the
approximating solution $u$ satisfies
\beq\label{ns2}
	\frac{du}{dt}+\nu Au+B(u,u)=f
\wwords{for} t\in (t_n,t_{n+1}),
\eeq
where $u(t_n)=u_n$ is the initial condition given by
	$$u_0=JU_0\wwords{and} u_{n+1}=ES(t_{n+1},t_n;
u_n)+JU(t_{n+1}).$$
Note that $u_n\in\DA\subseteq V$, for $n=0,1,2,\dots$.
Consequently, the solution of \eqref{ns2}
with initial data $u(t_n)=u_n$
on the interval $(t_n, t_{n+1})$
is a strong solution of the Navier--Stokes equations.
Moreover, because of our working assumptions on $f$ we further
obtain that
$u(t)\in\DA$ for $t\in[t_n,t_{n+1}]$.
It follows that the estimates we make in the proof of
Proposition \ref{tildelemma} below, and in the results
which follow, are rigorous;
in particular, $v=U-u$ exists, is unique and $Av$ makes sense
at all times $t\ge t_0$.

The equations governing the evolution of $v$ may be written as
\beq\label{nse1}
	\frac{dv}{dt}+\nu Av+B(v,U)+B(U, v)+B(v,v)=0
\eeq
for $t\in(t_n, t_{n+1})$, with $v(t_n)=U(t_n)-u_n$, for $n=0,1,2,\dots.$.

\begin{proposition}\label{tildelemma}
Let $\tilde{v}_n=U(t_n)-S(t_n,t_{n-1};u_{n-1})$.
For every $\delta>0$ there are $\lambda$, large enough,
and $h$, small enough, for which there exists $\gamma\in (0,1)$
such that
$$
	\|\tilde{v}_{n+1}\|^2\le \gamma\|\tilde{v}_n\|^2,
	\wwords{for all} n=1,2,\ldots.
$$
\end{proposition}

\begin{proof}
Multiplying \eqref{nse1} by $Av$ and then integrating over $\Omega$ we have
$$
	\frac 12\frac{d}{dt} \|v\|^2
		+\nu  |Av|^2+ (B(v,U), Av)+(B(U, v), Av)+(B(v,v),Av)=0.
$$
By \eqref{balgp} and \eqref{borthp},  we have
\beq\label{me1}
	\frac 12\frac{d}{dt} \|v\|^2+\nu  |Av|^2=(B(v,v), AU).
\eeq
Estimate the term on the right of the previous equation.
Using \eqref{Best1} and then the interpolation inequality $\|v\|\le |v|^{1/2}|Av|^{1/2}$ yields
\begin{align*}
|(B(v,v), AU)| &\le c|v|^{1/2}\|v\|^{1/2} \|v\|^{1/2} |Av|^{1/2}|AU|\\
	&= c|v|^{1/2}\|v\||Av|^{1/2}|AU|\\
	&\le c|v|^{1/2}|v|^{1/2}|Av|^{1/2}|Av|^{1/2}|AU|\\
	&= c|v||Av||AU|.
\end{align*}
Combining this with \eqref{me1}, we have
\beq\label{me2n}
	\frac 12\frac{d}{dt} \|v\|^2+\nu  |Av|^2\le  c|v||Av||AU|.
\eeq
Now, apply Young's inequality to obtain
\beq\label{BwwAu}
	\frac{d}{dt} \|v\|^2+\nu  |Av|^2
		\le \frac{c^2}{\nu}|v|^2|AU|^2.
\eeq
From Poincar\'e's inequality \eqref{poincareDA}
followed by \eqref{wttn}, we get
\beq\label{me4}
	\frac{d}{dt} \|v\|^2+\lambda_1\nu \|v\|^2
		\le \frac{c^2}{\nu}|AU|^2e^{\beta(t-t_n)}|v_n|^2
		\le \frac{c^2}{\nu}|AU|^2e^{\beta \delta}|v_n|^2,
\eeq
where we have assumed $t\in [t_n,t_{n+1})$.
Multiply equation \eqref{me4} by
$e^{\lambda_1\nu t}$
and
then integrate in time from $t_n$ to $t$.  Thus,
\begin{plain}\begin{equation}\label{me8}\eqalign{
	\|v(t)\|^2
	&\le e^{-\lambda_1\nu (t-t_n)}\|v_n\|^2
	+\frac{c^2}{\nu}
	e^{-\lambda_1\nu(t-t_n)+(\beta+\lambda_1\nu)\delta}
		|v_n|^2\int_{t_n}^t|AU(s)|^2ds
}\end{equation}\end{plain}%
for $t\in [t_n, t_{n+1})$.
Combining \eqref{me8} with the {\it a priori\/} estimate
\eqref{Aint}, we have
\beq\label{me8n}
\|v(t)\|^2\le e^{-\lambda_1\nu (t-t_n)}\|v_n\|^2
	+\frac{c^2\rho_V^2}{\nu}
	\left( \frac{1}{\nu}+{\delta \lambda_1\over 2}\right)
	e^{-\lambda_1\nu(t-t_n)+(\beta+\lambda_1\nu)\delta}
	|v_n|^2.
\eeq
Since $n\ge 1$ then
\begin{align*}
	v_n=U(t_n)-u_n&=U(t_n)-ES(t_n,t_{n-1};u_{n-1})-JU(t_n)\\
	&=E\big(U(t_n)-S(t_n,t_{n-1};u_{n-1})\big)
	=E(\tilde v_n),
\end{align*}
and by Proposition \ref{Ebounds},
we can estimate
\begin{align*}
	|v_n|^2
		\le \lambda^{-1}(1+\varepsilon)\|\tilde{v}_n\|^2
\wwords{and}
	\|v_n\|^2
	\le (1+\varepsilon)\|\tilde{v}_n\|^2.
\end{align*}
Hence \eqref{me8n} becomes
\beq\label{tbound}
\|v(t)\|^2\le
	(1+\varepsilon)
e^{-\lambda_1\nu (t-t_n)}
\left[1+\frac{c^2\rho_V^2}{\lambda\nu}
		\left( \frac{1}{\nu}+{\delta \lambda_1\over 2}\right)
	e^{(\beta+\lambda_1\nu)\delta}
		\right]\|\tilde{v}_n\|^2,
\eeq
for $t\in [t_n, t_{n+1})$.
Taking the limit as $t \nearrow t_{n+1}$ results in
$
	\|\tilde{v}_{n+1}\|^2 \le\gamma
		\|\tilde{v}_n\|^2,
$
where
$$	\gamma= (1+\varepsilon)
		\left[e^{-\lambda_1\nu \delta}+\frac{c^2\rho_V^2}{\lambda\nu}
	\left( \frac{1}{\nu}+{\delta \lambda_1\over 2}\right)
	e^{\beta\delta}
	\right].
$$
We now show for every $\delta>0$ that there exists $\lambda$
and $h$ such that $\gamma\in(0,1)$.  First, since
$$
	e^{-\lambda_1\nu\delta}<1\wwords{and}
	\frac{c^2\rho_V^2}{\lambda\nu}
		\left( \frac{1}{\nu}+{\delta \lambda_1\over 2}\right)
		e^{\beta \delta}
		\to 0
	\words{as}\lambda\to\infty,
$$
then there is $\lambda$ large enough such that
$$
		e^{-\lambda_1\nu \delta}+\frac{c^2\rho_V^2}{\lambda\nu}
	\left( \frac{1}{\nu}+{\delta \lambda_1\over 2}\right)
	e^{\beta\delta}
	<1.
$$
Finally, since $\varepsilon\to 0$, as $h\to 0$, while holding
$\lambda$ fixed,
then there is $h$ small enough such that $1+\varepsilon$
is small enough to ensure that $\gamma<1$.
\end{proof}
Observe that by a more careful analysis one could find explicit
choices for $\lambda$ and $h$ in terms of $\beta$, $\delta$,
$\lambda_1$, $\nu$ and $\rho_V$.
Note also that there is a dependency between $\lambda$ and $h$.
Since $h$ is a physical parameter related to the resolution
of the observations
while $\lambda$ is an easily-adjusted parameter
related to our spectral filter, it would be reasonable to further
choose $\lambda$ to minimize $h$.
The resulting estimate on $h$ could then be used to compare
the sharpness of the above theoretical bounds to alternative approaches
to the analysis, to numerical results obtained from simulation and to
similar analysis for different data assimilation schemes.
Such comparisons, while interesting, are outside the scope of
the present work.  We end this section with our main result on
type-I interpolant observables.

\begin{theorem}\label{mainI}
If $\delta$, $h$ and $\lambda$ are chosen appropriately as in Proposition \eqref{tildelemma},
then $\|U(t)-u(t)\|\to 0$, as $t\to\infty$.
Moreover, the rate of convergence is exponential in time.
\end{theorem}
\begin{proof}
Choose $\delta$, $h$ and $\lambda$ as in Proposition \ref{tildelemma}.
In reference to equation \eqref{tbound}, let
$$
	M=(1+\varepsilon)
\left[1+\frac{c^2\rho_V^2}{\lambda\nu}
		\left( \frac{1}{\nu}+{\delta \lambda_1\over 2}\right)
	e^{(\beta+\lambda_1\nu)\delta}
		\right].
$$
We first bound $\tilde v_1$ in terms of $v_0$.
Since
$$
	v_0=U_0-u_0=U_0-JU_0=EU_0,
$$
then Proposition \ref{Ebounds} and the working
assumptions which follow Proposition \ref{attbnd} 
	yield that
$$
	|v_0|^2=|EU_0|^2 \le
		\lambda^{-1} (1+\varepsilon)\|U_0\|^2
		\le
		\lambda^{-1} (1+\varepsilon)\rho_V^2,
$$
and similarly that $\|v_0\|^2\le (1+\varepsilon)\rho_V^2$.
These two bounds substituted into \eqref{me8n}
for $n=0$ imply
\beq\label{me9n}
\|v(t)\|^2\le (1+\varepsilon)
e^{-\lambda_1\nu (t-t_0)}
		\left[1
	+\frac{c^2\rho_V^2}{\lambda\nu}
	\left( \frac{1}{\nu}+{\delta \lambda_1\over 2}\right)
	e^{(\beta+\lambda_1\nu)\delta}
	\right]
	\rho_V^2.
\eeq
for $t\in[t_0, t_1)$.
Taking the limit as
$t\nearrow t_1$ results in $\|\tilde v_1\|^2\le\gamma\rho_V^2$
where $\gamma\in(0,1)$.

Now, given $t>0$ choose $n$ such that $t\in [t_n,t_{n+1})$.
Since $n>(t-t_0)/\delta-1$,
it follows from (\ref{tbound}) that
$$\|U(t)-u(t)\|^2=\|v(t)\|^2\le M\|\tilde v_n\|^2
	\le M\gamma^{n} \rho_V^2
	\le M \gamma^{-1} \rho_V^2 e^{-\alpha (t-t_0)},
$$
where
$\alpha=\delta^{-1}\log(\gamma^{-1})$.
Note that $\gamma\in(0,1)$ implies $\alpha>0$.
It follows that $\|U(t)-u(t)\|$
converges to zero at an exponential rate.
\end{proof}

\section{Type-II Interpolant Observables}
In this section we treat the case when $I_h$ is a type-II
interpolant observable.
As before let $v=U-u$ where $U$ is the exact solution
to \eqref{2dns} about which we know only limited information
through the observables
and $u$ is the approximating process obtained by the
spectrally-filtered discrete data assimilation algorithm
given in Definition \ref{algorithm}.
The proof that the difference between $u$ and $U$ decays to zero
over time is complicated by the fact that the $|Av|$
norm enters into the bounds given by Proposition \ref{Ebounds}
and therefore needs to be controlled.
To do so, we shall employ an equation similar to \eqref{me1}
which governs the evolution of $|Av|^2$.
While such an equation could be obtained by formally multiplying
\eqref{nse1} by $A^2u$ and integrating over $\Omega$,
it is easier to work with the vorticity in two-dimensions.

Let $W=\curl U$, $w=\curl u$, and $g=\curl f$ where
$\curl$ has been defined such that
\begin{plain}$$
	\curl \Phi=
		{\partial \Phi_2(x_1,x_2)\over\partial x_1}
		-{\partial \Phi_1(x_1,x_2)\over\partial x_2}
\wwords{when}
	\Phi(x)=\big(\Phi_1(x_1,x_2),\Phi_2(x_1,x_2)\big).
$$\end{plain}%
Since $u$ is the approximating solution
described in Definition \ref{algorithm},
then $w$ is the resulting vorticity approximation of $W$.
Written in terms of vorticity, the corresponding version of Theorem
\ref{mainI} for type-II interpolant observables is given by
\begin{theorem}\label{mainII}
If $\delta$, $h$ and $\lambda$ are chosen appropriately,
then $\|W-w\|\to 0$, as $t\to\infty$.
Moreover, the rate of convergence is exponential in time.
\end{theorem}

Before proving Theorem \ref{mainII} we fix our notation
by stating a few facts
about the vorticity and proving a lemma containing
bounds for non-linear terms that will be used later.
First note, after taking the $\curl$ of \eqref{navierstokes},
that Definition \ref{algorithm} implies $W$ and $w$ satisfy
\beq\label{nset}
	\frac{\partial W}{\partial t}-\nu \Delta W+(U\cdot\nabla)W=g
\wwords{and}
	\frac{\partial w}{\partial t}-\nu \Delta w+(u\cdot\nabla)w=g
\eeq
on each interval $(t_n,t_{n+1})$.  Our working assumptions
in the case of type-II interpolant observables ensure that
the equations (\ref{nset}) hold in the strong sense.
In particular,
$W=\curl U$ and $w=\curl u$ exist and
$|\Delta W|=|A^{3/2} U|$ and
$|\Delta w|=|A^{3/2} u|$
are finite almost everywhere.
Therefore, the equations governing the evolution through
the vorticity of the difference $\xi=W-w$ may be written as
\beq\label{nset2}
	\frac{\partial\xi}{\partial t}-\nu \Delta \xi+(v\cdot\nabla)W
		+(v\cdot\nabla)\xi+(U\cdot\nabla)\xi=0,
\eeq
where $\xi(t_n)=W(t_n)-\curl u_n$ and $v=\curl^{-1}\xi$.

Since $v$ is divergence-free with zero average, then
$\curl^{-1}\xi$ is well defined and may be written in terms of
Fourier series as
\begin{plain}$$
	\curl^{-1}\xi=
	\sum_{k\in\J} {i(k_2,-k_1)\over |k|^2} \widehat\xi_k e^{ik\cdot x}
\wwords{when}
	\xi=\sum_{k\in\J} \widehat\xi_k e^{ik\cdot x}.
$$\end{plain}%
Recall that the divergence-free condition $k\cdot\widehat v_k=0$ implies
\begin{plain}$$\eqalign{
	|\widehat\xi_k|^2&=
		|ik_1\widehat v_{k,2}-ik_2\widehat v_{k,1}|^2
		=
		k_1^2 |\widehat v_{k,2}|^2+k_2^2 |\widehat v_{k,1}|^2
		- k_1k_2 \widehat v_{k,1}\widehat v_{k,2}^*
		- k_1k_2 \widehat v_{k,1}^*\widehat v_{k,2}\cr
		&=
		k_1^2 |\widehat v_{k,2}|^2+k_2^2 |\widehat v_{k,1}|^2
		+ k_2^2 |\widehat v_{k,2}|^2
		+ k_1^2 |\widehat v_{k,1}|^2
		= |k|^2 |\widehat v_k|^2.
}$$\end{plain}%
Therefore
$$
	|\xi|^2 = L^2\sum_{k\in\J} |\widehat \xi_k|^2=\|v\|^2
\wwords{and}
	\|\xi\|^2 =L^2\sum_{k\in\J} |k|^2|\widehat \xi_k|^2 = |Av|^2.
$$

To keep the notation in the present section similar to the notation
appearing in the previous section, we abuse it by extending the
definitions of $B$ and $A$ to the vorticity as
$$
	B(v,\xi)=(v\cdot\nabla)\xi
\wwords{and}
	A\xi=-\Delta\xi.
$$
Thus, equation \eqref{nset2} may be written as
\begin{equation}\label{zetaeq}
	\frac{d\xi}{dt}+\nu A \xi+B(v,W)
		+B(v,\xi)+B(U,\xi)=0.
\end{equation}
Equations \eqref{zetaeq} are similar to \eqref{nse1} in structure;
however, there are no cancellations when multiplying by $A\xi$
and integrating over $\Omega$.  To bound the resulting terms
we prove
\begin{lemma}\label{lemmaEyes}
Let $U$, $W$, $v$ and $\xi$ be defined as above.
The following bounds hold
\begin{plain}$$\eqalignno{
	|(B(v, W), A\xi)|
	&\le C\frac{4^2}{3\nu^2}|v|^2\|W\|^3+\frac{\nu}{6}|A\xi|^2,\cr
	|(B(v,\xi), A\xi)|
    &\le C\frac{4^2}{3\nu^2}|v|^{2} \|\xi\|^3
        + \frac{\nu}{6}|A\xi|^{2},\cr
\noalign{\medskip\noindent and\medskip}
	|(B(U, \xi), A\xi)|
	&\le C\frac{5^5}{\nu^5}\|U\|^6_{L^\infty}|v|^{2}
        +\frac{\nu}{6}|A\xi|^2
}$$\end{plain}%
for almost every $t\ge t_0$.
\end{lemma}

\begin{proof}
The condition (\ref{vortpose}) applied to both $U$
and $u$ implies that $\|A v\|=|A\xi|$ is finite for almost every
$t\ge t_0$.  Our working assumptions further imply that the other norms
appearing in the above bounds
exist everywhere.
For convenience denote
\begin{equation*}
	I_1=|(B(v, W), A\xi)|,\qquad
	I_2=|(B(v,\xi), A\xi)|\wwords{and}
	I_3=|(B(U, \xi), A\xi)|.
\end{equation*}
We now estimate $I_1$, $I_2$ and $I_3$ in turn.
First, estimate $I_1$ using \eqref{Best0}
followed by Agmon's inequality to obtain
\begin{align*}
	I_1\le\|v\|_{L^\infty} \|W\| |A\xi|
	\le C|v|^{1/2}|Av|^{1/2}\|W\| |A\xi|
	=C|v|^{1/2}|Av|^{1/2-\theta}|Av|^{\theta}\|W\| |A\xi|.
\end{align*}
Since $|Av|=\|\xi\|$, we have
\begin{align*}
	I_1\le C|v|^{1/2}\|\xi\|^{1/2-\theta}|Av|^{\theta}\|W\| |A\xi|.
\end{align*}
We now use interpolation inequality on $|Av|^{\theta}$
and have $|Av|^{\theta}\le |v|^{\theta/3}|A\xi|^{2\theta/3}$.
This yields
\begin{align*}
	I_1\le C|v|^{1/2+\theta/3}\|\xi\|^{1/2-\theta}
		|A\xi|^{1+2\theta/3}\|W\|.
\end{align*}
Using Young's inequality with powers $3$ and $3/2$, we have
\begin{align*}
	I_1\le C\frac{16}{3\nu^2}|v|^{3/2+\theta}\|\xi\|^{3/2-3\theta}
		\|W\|^3+\frac{\nu}{6}|A\xi|^{3/2+\theta}.
\end{align*}
Choose $\theta=\frac{1}{2}$, then we have
\beq\label{estIone}
	I_1\le C\frac{16}{3\nu^2}|v|^2\|W\|^3+\frac{\nu}{6}|A\xi|^2.
\eeq
Next, estimate $I_2$ using \eqref{Best0}
and then Agmon's inequality.  We have
\begin{align*}
	I_2\le C|v|^{1/2} |Av|^{1/2} \|\xi\| |A\xi|
		=C|v|^{1/2} |Av|^{3/2-\theta} |Av|^{\theta}|A\xi|.
\end{align*}
Using interpolation on $|Av|^{\theta}$ it follows that
\begin{align*}
	I_2\le C|v|^{1/2+\theta/3} |Av|^{3/2-\theta} |A\xi|^{1+2\theta/3}.
\end{align*}
Choosing $\theta=\frac 12$ and then by Young's inequality
with powers $3$ and $3/2$, we have
\beq\label{estItwo}
	I_2\le C|v|^{4/6} |Av| |A\xi|^{4/3}
		\le C\frac{4^2}{3\nu^2}|v|^{2} \|\xi\|^3
		+ \frac{\nu}{6}|A\xi|^{2}.
\eeq
Finally, estimate $I_3$ using \eqref{Best0}.  We have
\begin{align*}
	I_3&\le \|U\|_{L^\infty}\| \xi\| |A\xi|
		=\|U\|_{L^\infty}| Av| |A\xi|\\
	& \le \|U\|_{L^\infty}|v|^{1/3}|A\xi|^{5/3}.
\end{align*}
Using Young's inequality with powers $6$ and $6/5$ it follows that
\beq\label{estIthree}
	I_3\le C\frac{5^5}{\nu^5}\|U\|^6_{L^\infty}|v|^{2}
		+\frac{\nu}{6}|A\xi|^2.
\eeq
\end{proof}

\begin{proof}[Proof of Theorem \ref{mainII}]
Multiplying equations \eqref{zetaeq} by $A\xi$
and integrating over $\Omega$ yields
\begin{equation}
\label{es1}
	\frac 12\frac{d}{dt}\|\xi\|^2+\nu |A\xi|^2
		+(B(v, W), A\xi)+(B(v,\xi), A\xi)+(B(U, \xi), A\xi)=0.
\end{equation}
We remark that the working assumptions for type-II interpolant
observables imply both
$U$ and $u$ and consequently their difference has the needed
regularity for
the above equation to make sense.  These assumptions further provide
{\it a priori\/} bounds on $U$ which are uniform in time.
Although
the corresponding norms of $u$ are finite, we cannot at this point
assume they are uniformly bounded in time.
Under the hypotheses of this theorem, however,
uniform bounds on $u$
can be inferred
as a consequence of this proof

Now, plug the estimates given by Lemma \ref{lemmaEyes} into
\eqref{es1} to obtain
\beq\label{es2}
	\frac{d}{dt}\|\xi\|^2+ \nu|A\xi|^2
	\le C\left(\frac{1}{\nu^2}\|W\|^3
		+\frac{1}{\nu^2}\|\xi\|^3
		+ \frac{1}{\nu^5}\|U\|^6_{L^\infty}\right)|v|^{2}.
\eeq
We again point out that $C$ is a non-dimensional constant independent
of $\delta$, $\lambda$ and $h$.
By Poincar\'e's inequality \eqref{poincareDA} we have
\begin{plain}\beq\label{es3}
	\frac{d}{dt}\|\xi\|^2+ \lambda_1\nu\|\xi\|^2\le
	{C\over \nu^2}\left(\|\xi\|^3+K\right)|v|^{2},
\eeq\end{plain}%
where the assumption $\|W\|\le\rho_A$ combined
with Agmon's inequality \eqref{agmon} allows us to take
$$K=\rho_A^3(1+ c \nu^{-3}\rho_H^3).$$

Alternatively, one could write $K'=\rho_A(1+ c \nu^{-3}\rho_H^3)$
to obtain
\begin{plain}$$
	\frac{d}{dt}\|\xi\|^2+ \lambda_1\nu\|\xi\|^2\le
	{C\over \nu^2}\left(\|\xi\|^3+K'|AU|^2\right)|v|^{2},
$$\end{plain}%
and then estimate the integral of $|AU|^2$ using \eqref{Aint}
as we did in \eqref{me8n}.
Unfortunately, this improvement is dominated by
subsequent estimates on $\|\xi\|$ which are
proportional to $\rho_A$.
Therefore, as the differences are minimal,
we continue with \eqref{es3} for simplicity.

By \eqref{wttn}, we have
\begin{plain}\begin{align}
\label{es4}
	\frac{d}{dt}\|\xi\|^2&+ \lambda_1\nu\|\xi\|^2
	\le {C\over\nu^2}\left(\|\xi\|^3 +K\right)
		e^{\beta(t-t_n)}|v_n|^2.
\end{align}\end{plain}%
Note that equation \eqref{es4} is
similar to \eqref{me4} except for the additional term involving
$\|\xi\|^3$ on the right.
Fortunately, this term can be controlled for times of size
$\delta$ by our choosing~$h$ small and $\lambda$ large.
This complicates the proof and is the main reason why
the type-I interpolant observables were treated separately
in the previous section.

Continue as in the type-I case.  First,
multiply \eqref{es4} by $e^{\lambda_1\nu t}$,
integrate from $t_n$ to $t$
and simplify as in \eqref{me8} to obtain
\begin{align*}
\label{es5}
	\|\xi\|^2&\le\|\xi_n\|^2e^{-\lambda_1\nu (t-t_n)}
	+
	\frac {C}{\nu^2 \beta}\Big(\sup_{s\in[t_n,t)}\|\xi(s)\|^3
		+K\Big)e^{\beta(t-t_n)}|v_n|^2.
\end{align*}
When $n=0$ it follows from Proposition \ref{Ebounds} that
$$
	|v_0|^2=|E U_0|^2
	\le (\lambda\lambda_1)^{-1}(1+\varepsilon) |AU_0|^2
	\le (\lambda\lambda_1)^{-1}(1+\varepsilon) \rho_A^2,
$$
and
$$
	\|\xi_0\|^2=|A v_0|^2
		=|A EU_0|^2
		\le (1+\varepsilon) |AU_0|^2
		\le (1+\varepsilon) \rho_A^2.
$$
Therefore when $t\in[t_0,t_1)$ we have
\begin{equation}\label{zetaineq}
	\|\xi\|^2\le
	(1+\varepsilon)\bigg\{
	e^{-\lambda_1\nu (t-t_0)}+
	\frac C{\lambda\lambda_1\nu^2\beta}
		\Big(\sup_{s\in[t_0,t)}\|\xi(s)\|^3
		+K\Big)e^{\beta(t-t_0)}\bigg\}
	\rho_A^2.
\end{equation}
Let $\delta>0$ be arbitrary and define
$$
	\gamma=
	(1+\varepsilon)\bigg\{
	e^{-\lambda_1\nu \delta}+
	\frac C{\lambda\lambda_1\nu^2\beta}
		\Big(8\rho_A^3
		+K\Big)e^{\beta \delta}\bigg\}.
$$
As in the the proof of Proposition \ref{tildelemma}, since
$$
	e^{-\lambda_1\nu\delta}<1\wwords{and}
	\frac C{\lambda\lambda_1\nu^2\beta}
		\Big(8\rho_A^3
		+K\Big)e^{\beta \delta}\to 0\words{as}\lambda\to\infty,
$$
then there is $\lambda$ large enough such that
\begin{equation}\label{boundtwo}
	e^{-\lambda_1\nu\delta}+
	\frac C{\lambda\lambda_1\nu^2\beta}
		\Big(8\rho_A^3
		+K\Big)e^{\beta \delta}
	<1.
\end{equation}
Furthermore, since $\varepsilon\to0$ as $h\to0$
while holding $\lambda$ fixed,
then there is $h$ small enough such that $1+\varepsilon<2$
and moreover small enough to ensure that $\gamma<1$.

For the choice of $\delta$, $h$ and $\lambda$ given above, let
$$
	M= \sup_{s\in [0,\delta]}(1+\varepsilon)\bigg\{
	e^{-\lambda_1\nu s}+
	\frac C{\lambda\lambda_1\nu^2\beta}
		\Big(8\rho_A^3
		+K\Big)e^{\beta s}\bigg\},
$$
and note \eqref{boundtwo} along with the
fact that $1+\varepsilon<2$ implies $M<4$.
We claim that $\|\xi(s)\|< 2\rho_A$
for $s\in[t_0,t_1)$.
For contradiction, suppose not.
Since $\|\xi\|$ is continuous on $[t_0,t_1)$ and
$$
	\|\xi(t_0)\|=\|\xi_0\|\le (1+\varepsilon)^{1/2}\rho_A <
		2^{1/2}\rho_A<2\rho_A,
$$
then this would imply the existence of $t_*\in (t_0,t_1)$
such that
$$\|\xi(t_*)\|=2\rho_A\wwords{and}
  \|\xi(s)\|<2\rho_A\words{for} s\in [t_0,t_*).$$
However, if this were true, then inequality \eqref{zetaineq} would imply
$$
\|\xi(t_*)\|^2\le
	(1+\varepsilon)\bigg\{
	e^{-\lambda_1\nu(t_*-t_0)} +
	\frac C{\lambda\lambda_1\nu^2\beta}
		\Big(8\rho_A^3
		+K\Big)e^{\beta (t_*-t_0)}\bigg\}\rho_A^2
	\le M\rho_A^2 < 4\rho_A^2,
$$
which is a contradiction.
Therefore $\|\xi(s)\|< 2\rho_A$ for $s\in[t_0,t_1)$.
Consequently
$$
	\sup_{s\in [t_0,t_1)}\|\xi(s)\|^3 \le 8\rho_A^3,
$$
and taking the limit of \eqref{zetaineq} as $t\to t_1$ results in
	$\|\tilde \xi_1\|^2\le\gamma \rho_A^2$.

We proceed by induction.
Let $n\ge 1$ and suppose
$$\|\tilde \xi_n\|^2\le \gamma^n\rho_A^2.$$
By Proposition \ref{Ebounds} it follows that
$$
	|v_n|^2=|E\tilde v_n|^2
		\le (\lambda\lambda_1)^{-1} (1+\varepsilon) |A\tilde v_n|^2
		= (\lambda\lambda_1)^{-1} (1+\varepsilon) \|\tilde\xi_n\|^2,
$$
and
$$
	\|\xi_n\|^2=|A v_n|^2
		=|AE\tilde v_n|^2
		\le (1+\varepsilon) |A\tilde v_n|^2
		= (1+\varepsilon) \|\tilde\xi_n\|^2,
$$
where $\tilde \xi_n=\curl\tilde v_n$.
Since $1+\varepsilon<2$ we obtain
$$\|\xi(t_n)\|=\|\xi_n\|\le (1+\varepsilon)^{1/2}\|\tilde \xi_n\|
		\le 2^{1/2} \gamma^{n/2} \rho_A < 2\rho_A.
$$
Following the same arguments as before, we obtain that
$$
	\sup_{s\in[t_n,t_{n+1})} \|\xi(s)\|^3 \le 8 \rho_A^3,
$$
and taking limits as $t\nearrow t_{n+1}$ conclude that
$$\|\tilde \xi_{n+1}\|^2\le\gamma \|\tilde\xi_n\|^2
	\le\gamma^{n+1} \rho_A^2,
$$
which completes the induction.

Given $t>0$ choose $n$ such that $t\in[t_n,t_{n+1})$.
It follows that
$$
	\|W-w\|^2=\|\xi\|^2\le M \|\tilde \xi_n\|^2
		\le M\gamma^n\rho_A^2\to 0\wwords{as} t\to\infty.
$$
Therefore, the same argument used in the proof of Theorem
\ref{mainI} now implies
$$\|W-w\|\to 0\wwords{exponentially as}t\to\infty,$$
and finishes the proof of Theorem \ref{mainII}.
\end{proof}

\section{Conclusions}
In this paper we have shown that spectrally-filtered discrete
data assimilation as described in Definition \ref{algorithm}
results in an approximating solution $u$ that converges to the
reference solution $U$ over time
for any general interpolant observable of type-I or type-II
when $\delta$, $\lambda$ and $h$ are chosen appropriately.
In particular, when observations of $U$ are made using nodal
points of the velocity field, we obtain a type-II interpolant
observable which our analysis is able to handle.
We note that this analysis relies crucially on properties of
the spectral filter and would not have been possible if the
unfiltered interpolants were used instead.
Specifically, our analysis makes use of the fact that
the filtered interpolant $E$ can be made to have norm near
unity when viewed as linear operator on the functional space
implied by the bounds on the original interpolant.
This fact is characterized by the respective inequalities
$$
	\|EU\|^2\le (1+\varepsilon)\|U\|^2
\wwords{and}
	|AEU|^2\le (1+\varepsilon)|AU|^2,
$$
for the type-I and type-II interpolant
observables given in Proposition \ref{Ebounds}.
Different filtering methods which satisfy similar inequalities
should also be effective.
As a number of advances in practical data assimilation have resulted
from improved filtering, we find
these analytic results to be interesting and relevant.

While it may seem anticlimactic that the technique crucial
for our analysis relies on spectrally projecting the interpolant
observable in Fourier space,
since the linear term
is responsible for the dissipation, it is natural that a spectral basis 
with respect to that linearity provides a convenient framework in which 
to analyze the synchronization properties of our data assimilation
algorithm.  Furthermore, using this basis as a means of spatial filtering
not only has the advantage of being simple, but is intrinsically
compatible with the reliance of our analysis on the dissipation.

Note that the functional dependency of $h$ and $\lambda$ on
$\delta$ and the other physical parameters in the system
appearing in Theorem \ref{mainresult}
depend on knowing an {\it a priori\/} bound $\rho_A$ on the norm
$|AU|$ in terms of those other parameters.
While suitable theoretical bounds appear in the literature,
these bounds are, in general, not sharp compared to
{\it a posteriori\/} bounds obtained through
numerical simulation.
Moreover, the algorithm may continue to work with values of $h$ much larger
and values of $\lambda$ much smaller than required by our analysis.
For example, computational experiments performed by \cite{Gesho2016}
for a different spatially filtered continuous data assimilation method
based on nudging show that the method performs far better
than the analytical estimates suggest.
We conjecture similar numerical effectiveness for the discrete
data assimilation method described in the present paper.
Therefore, we refrain from determining an
explicit theoretical relation between $h$ and the Grashof number
in this work, though such could be obtained from our
analysis, and save such comparisons for the context of a
future numerical study.

\section*{Acknowledgments}
The work of Eric Olson was supported in part
by NSF grant DMS-1418928.
The work of Edriss S. Titi was supported
in part by ONR grant N00014-15-1-2333, the Einstein Stiftung/Foundation - Berlin, through the Einstein Visiting Fellow Program, and by the John Simon Guggenheim Memorial Foundation.

\bibliographystyle{abbrv}

\vfill\eject

\appendix
\section{Estimates for Time Dependent Forcing}\label{apriori}
In this appendix we present {\it a priori\/} estimates on the
solution $U$ to the two-dimensional incompressible Navier--Stokes
equations (\ref{2dns}) in the case where the body force $f$
depends on time.
While these results are straight forward, we could not find
suitable references in the literature and have therefore included
them here for completeness of our presentation.
Note that the first bound, stated as Theorem \ref{apriori1} below,
will be sufficient for our analysis in the case of type-I interpolant
operators.
The second bound, Theorem \ref{apriori2}, will be used
for type-II interpolant operators.

In addition to the facts and inequalities from Section \ref{secprelim}
this appendix makes use of Ladyzhenskaya's inequality, which in
two-dimensions interpolates $L^4$ as
\begin{equation}\label{lady}
	\|U\|_{L^4} \le C_0 |U|^{1/2} \|U\|^{1/2},
\end{equation}
where $C_0$ is a non-dimensional constant depending only on $\Omega$.
We also make use of the following $L^2$ and $H^1$ bounds on
the nonlinear term.
\begin{lemma}\label{bh1}
If $U\in V$ then
\begin{equation}\label{lemBp1}
	|B(U,U)|\le C_0^2 |U|^{1/2} \|U\| |AU|^{1/2}.
\end{equation}
Furthermore, if $U\in\DA$ then
\begin{equation}\label{lemBp2}
	\|B(U,U)\|\le C_1 \|U\||AU|+ C_2 |U|^{1/2}|AU|^{3/2}.
\end{equation}
Here $C_0$, $C_1$ and $C_2$ are non-dimensional constants
depending only on $\Omega$.
\end{lemma}
\begin{proof}
Given $U\in V$ apply (\ref{lady}) to obtain
$$
	|B(U,U)|\le \|U\|_{L^4} \|\nabla U\|_{L^4}
		\le C_0^2 |U|^{1/2} \|U\| |AU|^{1/2},
$$
which is the first inequality.

Now suppose that $U\in\DA$. Define
$$
	\Psi_\alpha=\sum_{k\in\J} |k|^\alpha|\widehat U_k| e^{ik\cdot x}
\wwords{where}
	U = \sum_{k\in\J} \widehat U_k e^{ik\cdot x}.
$$
Further define $\J_0=\J\cup\{ (0,0)\}$ and recall the notational convention
that $\widehat U_0=0$.
Note that $\|\Psi_\alpha\|_{L^2}=\|U\|_\alpha$ for all $\alpha\le 2$.
Moreover, (\ref{agmon}) and (\ref{lady}) imply that
$$
	\|\Psi_0\|_{L^\infty}\le C \|\Psi_0\|_{L^2}^{1/2}
			\|\Psi_2\|_{L^2}^{1/2}
\wwords{and}
	\|\Psi_1\|_{L^4}\le C_0 \|\Psi_1\|_{L^2}^{1/2}
			\|\Psi_2\|_{L^2}^{1/2}.
$$
Since
$$
	U\cdot\nabla U=i
		\sum_{k,\ell\in\J}
		(\widehat U_k\cdot\ell)\widehat U_\ell\, e^{i(k+\ell)\cdot x}
		=i \sum_{\ell,m\in\J_0}
		\big(\widehat U_{m-\ell}\cdot\ell\big)\widehat U_\ell\, e^{i m\cdot x},
$$
it follows that
\begin{align*}
	\|B(U,U)\|^2&\le
		L^2\sum_{m\in\J_0} |m|^2\Big|
		\sum_{\ell\in\J_0}
		\big(\widehat U_{m-\ell}\cdot\ell\big)\widehat U_\ell
		\Big|^2
		\le
		L^2\sum_{m\in\J_0} \Big|
		\sum_{\ell\in\J_0}
		|m||\widehat U_{m-\ell}||\ell||\widehat U_\ell|
		\Big|^2\\
		&\le
		2L^2\sum_{m\in\J_0} \Big|
		\sum_{\ell\in\J_0}
		|m-\ell||\widehat U_{m-\ell}||\ell||\widehat U_\ell|\Big|^2
		+
		2L^2\sum_{m\in\J_0} \Big|
		\sum_{\ell\in\J_0}
		|\widehat U_{m-\ell}||\ell|^2|\widehat U_\ell|
		\Big|^2\\
		&= 2\|\Psi_1^2\|_{L^2}^2+2\|\Psi_0\Psi_2\|_{L^2}^2
			\le 2 \|\Psi_1\|_{L^4}^4
				+2 \|\Psi_0\|_{L^\infty}^2 \|\Psi_2\|_{L^2}^2\\
		&\le 2C_0^4 \|\Psi_1\|_{L^2}^2\|\Psi_2\|_{L^2}^2
				+ 2 C^2 \|\Psi_0\|_{L^2}\|\Psi_2\|_{L^2} \|\Psi_2\|_{L^2}^2\\
		&=2C_0^4\|U\|^2|AU|^2+2C^2|U||AU|^3.
\end{align*}
Taking $C_1=\sqrt{2}C_0^2$ and $C_2=\sqrt{2}C$ finishes the proof of the lemma.
\end{proof}

\begin{theorem}\label{apriori1}
Suppose
$f\in L^{\infty}([t_0,\infty);H)$
is time-dependent and define
$$
	F=\esssup\big\{|f(t)|^2: t\in[t_0,\infty)\big\}.
$$
Then there are absorbing sets in $H$ and $V$
of radiuses $\rho_H$ and $\rho_V$, respectively, depending
only on $F$, $\Omega$ and $\nu$ such that for every $U_0\in H$
there is a time $t_V$ depending further on $|U_0|$
and $t_0$ for which
$$
	|U(t)|\le\rho_H \words{and} \|U(t)\|\le\rho_V
	\wwords{for all} t\ge t_V.
$$
Moreover,
$$
	\int_t^{t+\delta} |AU(s)|^2 ds \le
		\Big({1\over\nu}+{\delta\lambda_1\over 2}\Big)\rho_V^2
\wwords{for all} t\ge t_V.
$$
\end{theorem}

\begin{proof}  The proof is essentially the same as the
time-independent case appearing in \cite{Temam1983},
\cite{Constantin1988} or \cite{Robinson2001}
with $F$ is substituted for $|f|$ throughout.
For sake of brevity we present formal estimates
which could be rigorously justified
by means of the Galerkin method if desired.

First, take inner product of (\ref{2dns}) with $U$
and apply Cauchy's inequality followed by Young's inequality
to obtain
$$
	{1\over 2}{d\over dt} |U|^2+\nu\|U\|^2\le |f||U|
	\le
		{\lambda_1\nu\over 2}|U|^2
	+
		{1\over 2\lambda_1\nu}|f|^2.
$$
Collecting terms and applying the Poincar\'e inequality
(\ref{poincareV}) gives
\begin{equation}\label{deHnorm}
	{d\over dt}|U|^2+\nu\|U\|^2\le {F\over\lambda_1\nu}.
\end{equation}
Again applying (\ref{poincareV}), multiplying by
$e^{\lambda_1\nu t}$ and
integrating in time from $t_0$ to $t$ yields
$$
	|U(t)|^2\le e^{-\lambda_1\nu (t-t_0)} |U_0|^2
		+{F\over \lambda_1^2\nu^2} \big( 1-e^{-\lambda_1\nu(t-t_0)} \big).
$$
Upon taking $t_H$ so large that
$$e^{-\lambda_1\nu(t_H-t_0)}|U_0|^2
 	\le {F\over \lambda_1^2\nu^2},
$$
it follows that
$$
	|U(t)|\le\rho_H\words{for} t\ge t_H \wwords{where}
	\rho_H^2= {2F\over \lambda_1^2\nu^2}.
$$
Returning to (\ref{deHnorm}) for $t\ge t_H$ and simply
integrating both sides from $t$ to $t+\delta$ gives
$$
	|U(t+\delta)|^2-|U(t)|^2+\nu \int_t^{t+\delta} \|U(s)\|^2 ds
		\le {\delta F\over\lambda_1\nu}.
$$
Consequently,
\begin{equation}\label{intVbound}
	\int_t^{t+\delta} \|U(s)\|^2 ds
		\le {1\over \nu} |U(t)|^2 + {\delta\over\lambda_1\nu^2}F
		\le
\Big({1\over \nu} + {\delta\lambda_1\over 2} \Big)\rho_H^2.
\end{equation}

Second, take inner product of (\ref{2dns}) with $AU$ and apply
Cauchy's inequality followed by Young's inequality
to obtain
\begin{align*}
	{1\over 2}{d\over dt}\|U\|^2 +\nu|AU|^2
		\le |f||AU|
		\le {\nu\over 2} |AU|^2 + {1\over 2\nu} |f|^2.
\end{align*}
Collecting terms
gives
\begin{equation}\label{deVnorm}
	{d\over dt} \|U\|^2 + \nu |AU|^2 \le {F\over \nu}.
\end{equation}
Again applying (\ref{poincareDA}), multiplying by $e^{\lambda_1\nu t}$
and
integrating in time from $s$ to $t$ yields
$$
	\|U(t)\|^2	
		\le e^{-\lambda_1\nu(t-s)} \|U(s)\|^2
		+{F\over \lambda_1\nu^2}\big(
			1-e^{-\lambda_1\nu(t-s)}\big).
$$
Integrate with respect to $s$ from $t_H$ to $t_H+\delta$
using the fact that
$e^{-\lambda_1\nu(t-s)}\le e^{-\lambda_1\nu(t-t_H-\delta)}$
to obtain
\begin{align*}
	\delta\|U(t)\|^2
		&\le e^{-\lambda_1\nu(t-t_H-\delta)}
	\Big({1\over\nu}+{\delta\lambda_1\over 2}\Big)\rho_H^2
		+{F\delta\over \lambda_1\nu^2}.
\end{align*}
Setting $\delta=1/(\lambda_1\nu)$ yields
$$
	\|U(t)\|^2\le {3\lambda_1\over 2} e^{-\lambda_1\nu(t-t_H)+1}
		\rho_H^2
		+ {F\over\lambda_1\nu^2}.
$$
Upon taking $t_V\ge t_H$ so large that
$$
	{3\lambda_1\over 2}e^{-\lambda_1\nu(t_V-t_H)+1}\rho_H^2 \le {F\over \lambda_1\nu^2},
$$
it follows that
$$
	\|U(t)\|\le\rho_V\words{for}t\ge t_V
\wwords{where}
	\rho_V^2={2F\over \lambda_1\nu^2}.
$$
Returning to (\ref{deVnorm}) for $t\ge t_V$ and simply integrating
both sides from $t$ to $t+\delta$ gives
$$
	\|U(t+\delta)\|^2-\|U(t)\|^2
		+\nu\int_t^{t+\delta} |AU(s)|^2ds
		\le {F\delta\over\nu}.
$$
Consequently,
$$
	\int_t^{t+\delta} |AU(s)|^2ds \le {1\over\nu} \|U(t)\|^2
		+ {F\delta\over \nu^2}
		\le \Big({1\over \nu}+{\delta\lambda_1\over 2}\Big)\rho_V^2.
$$
This completes the proof.
\end{proof}

\begin{theorem}\label{apriori2}
Suppose $f\in L^{\infty}([t_0,\infty),V)$ and
$df/dt\in L^{\infty}([t_0,\infty),V^*)$ and define
$$
	G=\esssup\big\{ \|f(t)\|^2 : t\in [t_0,\infty)\big\}
\words{and}
	F_*=\esssup\big\{ \|df/dt\|_{-1}^2 : t\in [t_0,\infty)\big\}.
$$
Then there is an absorbing set in $\DA$ of radius $\rho_A$
depending only on $G$, $F^*$, $\Omega$ and $\nu$ such that for
every $U_0\in H$ there is a time $t_A$ depending further on
$|U_0|$ and $t_0$ for which
$$
	|AU(t)|\le \rho_A\wwords{for all} t\ge t_A.
$$
Moreover
$$
	\int_t^{t+1/(\lambda_1\nu)} \|AU(s)\|^2ds<  \infty
\wwords{for all} t\ge t_A.
$$
\end{theorem}
\begin{proof}
For convenience write $U_t=dU/dt$ and $f'=df/dt$.
We again present our estimates in a formal manner
with the remark that they could be rigorously justified if desired.

First, take inner product of (\ref{2dns}) with $U_t$ and apply
Agmon's inequality (\ref{agmon}) followed
by Young's inequality to obtain
\begin{align*}
|U_t|^2 &+{1\over 2}{d\over dt} \|U\|^2
		=-\big(B(U,U),U_t\big) + (f,U_t)
		\le \|U\|_{L^\infty} \|U\| |U_t| + |f||U_t|\\
	&\le C |U|^{1/2} \|U\| |AU|^{1/2} |U_t| + F^{1/2}|U_t|
	\le {1\over 2}|U_t|^2 + C^2 |U|\|U\|^2|AU| + F.
\end{align*}
Collecting terms, assuming $t\ge t_V$ and applying the results
of Theorem \ref{apriori1} yields
$$
	|U_t|^2+{d\over dt} \|U\|^2 \le 2C^2 \rho_H\rho_V^2|AU|+2F.
$$
Integrate from $t$ to $t+\delta$ and apply the Cauchy-Schwartz
inequality to obtain
\begin{align*}
	\int_t^{t+\delta} |U_t|^2 &+ \|U(t+\delta)\|^2
		\le \|U(t)\|^2+2C^2\rho_H\rho_V^2\int_t^{t+\delta}
			|AU(s)| ds
		+2\delta F\\
		&\le
		\rho_V^2+2C^2\delta^{1/2}\rho_H\rho_V^3
		\Big({1\over \nu}+{\delta\lambda_1\over 2}\Big)^{1/2}
		+2\delta F.
\end{align*}
Setting $\delta=1/(\lambda_1\nu)$ yields
$$
	\int_t^{t+1/(\lambda_1\nu)} |U_t(s)|^2 ds \le \sigma_H^2
\words{where}
	\sigma_H^2= \rho_V^2+
	{1\over \lambda_1\nu}\big(
	6^{1/2}C^2\lambda_1^{1/2}\rho_H\rho_V^3
		+2F\big).
$$

Second, differentiate (\ref{2dns}) with respect to $t$ to get
$$
	U_{tt} + \nu A U_{t} + B(U_t,U)+B(U,U_t) = f'.
$$
Take inner product with $U_t$, note the orthogonality
$\big(B(U,U_t),U_t\big)=0$ and apply Ladyzhenskaya's inequality
(\ref{lady}) followed by Young's inequality to obtain
\begin{align*}
	{1\over 2}{d\over dt}|U_t|^2&+\nu \|U_t\|^2
		=-\big(B(U_t,U),U_t\big)+(f',U_t)
		\le \|U_t\|_{L^4}^2 \|U\| + \|f'\|_{-1} \|U_t\| \\
	&\le C_0^2 |U_t| \|U_t\| \|U\|+ F_*^{1/2} \|U_t\|
	\le {\nu\over 2}\|U_t\|^2+ {C_0^4\over\nu}\|U\|^2|U_t|^2+ {F_*\over \nu}.
\end{align*}
Collecting terms yields
\begin{equation}\label{starstar}
	{d\over dt}|U_t|^2+\nu \|U_t\|^2
	\le \kappa\|U\|^2|U_t|^2+ {2F_*\over \nu}
\wwords{where}
	\kappa= {2C_0^4\over\nu}.
\end{equation}
Multiply by
$$
	\Phi(t)=\exp\Big(-\kappa\int_s^t \|U(\tau)\|^2d\tau\Big),
$$
and integrate from $s$ to $t+\delta$ to obtain
$$
	\Phi(t+\delta) |U_t(t+\delta)|^2 - |U_t(s)|^2 \le
		{2F_*\over\nu} \int_s^{t+\delta} \Phi(\sigma) d\sigma,
$$
or equivalently
\begin{align*}
	|U_t(t+\delta)|^2
		\le |U_t(s)|^2 &\exp\Big(\kappa\int_s^{t+\delta}
			\|U(\tau)\|^2d\tau\Big) \\
		&+{2F_*\over\nu}\int_s^{t+\delta}
			\exp\Big(\kappa\int_{\sigma}^{t+\delta}
				\|U(\tau)\|^2d\tau\Big)d\sigma.
\end{align*}
Integrate with respect to $s$ from $t$ to $t+\delta$.  Since
$t\ge t_V\ge t_H$ inequality (\ref{intVbound}) implies
\begin{align*}
	\delta |U_t(t+\delta)|^2
		&\le \int_t^{t+\delta}\Big\{
			|U_t(s)|^2 \exp\Big(\kappa\int_s^{t+\delta} \|U(\tau)\|^2
				d\tau\Big)\\
		&\qquad\qquad
			+{2F_*\over\nu}\int_s^{t+\delta}
            \exp\Big(\kappa\int_{\sigma}^{t+\delta}
				\|U(\tau)\|^2d\tau\Big)d\sigma
		\Big\}ds\\
		&\le \int_t^{t+\delta}\Big\{
			|U_t(s)|^2 \exp\Big(\kappa\int_t^{t+\delta} \|U(\tau)\|^2
				d\tau\Big)\\
		&\qquad\qquad
			+{2F_*\over\nu}\int_t^{t+\delta}
            \exp\Big(\kappa\int_t^{t+\delta} \|U(\tau)\|^2d\tau\Big)d\sigma
		\Big\}ds \\
		&\le
		\bigg(\int_t^{t+\delta} |U_t(s)|^2ds + {2F_*\delta^2\over\nu}\bigg)
			\exp\Big\{\kappa\Big({1\over\nu}+{\delta\lambda_1\over 2}\Big)\rho_H^2
				\Big\}.
\end{align*}
Setting $t_A=t_V+\delta$ with $\delta=1/(\lambda_1\nu)$ yields
$$
	|U_t(t)| \le R_H\words{for}t\ge t_A
\words{where}
		R_H^2=\Big(\lambda_1\nu\sigma_H^2
		+{2F_*\over\lambda_1\nu^2}\Big)
			\exp\Big({3C_0^4\rho_H^2\over\nu^2}\Big).
$$

We are now ready to estimate $|AU|$.
Upon taking $L^2$ norms
of (\ref{2dns}) and applying (\ref{lemBp1}) from Lemma \ref{bh1}
followed by Young's inequality we obtain
\begin{align*}
	\nu|AU|&\le |U_t|+|B(U,U)|+|f|
		\le |U_t|+C_0^2 |U|^{1/2}\|U\| |AU|^{1/2} + |f|\\
		&\le |U_t|+{C_0^4\over 2\nu} |U| \|U\|^2
			+ {\nu\over 2} |AU|^{1/2} +|f|.
\end{align*}
Therefore,
$$
	|AU|\le \rho_A\words{for}t\ge t_A \wwords{where}
		\rho_A={2\over \nu} R_H + {C_0^4\over \nu^2} \rho_H \rho_V^2
		+ {2F^{1/2}\over\nu}.
$$

To finish the proof, return to (\ref{starstar})
for $t\ge t_A$ and simply integrate
both sides from $t$ to $t+\delta$ to obtain
$$
	|U_t(t+\delta)|^2+\nu\int_t^{t+\delta} \|U_t\|^2
		\le |U_t(t)|^2 + \kappa \int_t^{t+\delta} \|U\|^2 |U_t|^2
			+{2\delta F_*\over \nu}.
$$
Setting $\delta=1/(\lambda_1\nu)$ and applying the previous
bounds for $t\ge t_A$ yields
$$
	\int_t^{t+1/(\lambda_1\nu)} \|U_t\|^2 \le \sigma_V^2
\wwords{where}
	\sigma_V^2
		={R_H^2\over\nu}
		+{3C_0^4\rho_H^2R_H^2 \over \nu^3}
			+ {2F_*\over\lambda_1\nu^3}.
$$
Now, upon taking $H^1$ norms of (\ref{2dns}) and applying
(\ref{lemBp2}) from Lemma \ref{bh1} followed by Young's inequality
we obtain
\begin{align*}
	\nu\|AU\| &\le \|U_t\|+ \|B(u,u)\|+\|f\| \\
		&\le
			\|U_t\|+C_1\|U\||AU|+C_2|U|^{1/2}|AU|^{3/2}+G^{1/2}.
\end{align*}
Consequently
$$
	\|AU\|^2 \le 4 \nu^{-2}\big(
		\|U_t\|^2+ C_1^2\|U\|^2|AU|^2+ C_2^2 |U||AU|^3 +G\big)
$$
implies
$$
	\int_t^{t+1/(\lambda_1\nu)} \|AU\|^2
		\le {4\sigma_V^2\over\nu^2}
			+ 4C_1^2 \Big({3\rho_V^2\over 2\nu^3}\Big)\rho_V^2
			+ 4C_2^2 \rho_H\rho_A\Big({3\rho_V^2\over 2\nu^3}\Big)
			+ {4G\over \lambda_1\nu^3}.
$$
Since the above bound is finite, this finishes the proof.
\end{proof}

\end{document}